\documentclass[12pt]{article}
\usepackage{amsmath}
\usepackage{amssymb}
\usepackage{amsthm,xspace}
\usepackage{graphicx,tikz}

\usepackage{fullpage}

\usepackage{mathrsfs}
\usepackage{bbm}
\usepackage{datetime}
\usepackage[numbers]{natbib}

\usepackage{color} 
\newcommand{\blue}[1]{\textcolor{blue}{#1}}
\newcommand{\red}[1]{\textcolor{red}{#1}}
\newcommand{\cyan}[1]{\textcolor{cyan}{#1}}

\theoremstyle{definition}
\newtheorem{definition}{Definition}
\newtheorem{remark}[definition]{Remark}
\newtheorem{example}[definition]{Example}

\theoremstyle{plain}
\newtheorem{theorem}[definition]{Theorem}

\newtheorem{proposition}[definition]{Proposition}

\newcommand{\R}{\mathbb{R}}
\renewcommand{\SS}{\mathbb{S}}
\newcommand{\E}{{\mathbb E}}
\renewcommand{\P}{{\mathbb P}}
\newcommand{\Q}{{\mathbf Q}}
\renewcommand{\H}{{\mathbf H}}
\newcommand{\hmg}{{\mathcal H}}
\newcommand{\eps}{\varepsilon}

\newcommand{\salg}{\mathfrak{K}}
\newcommand{\dsim}{\stackrel{d}{\sim}}
\newcommand{\thf}{\frac{1}{2}}
\newcommand{\Sphere}{\mathbb{S}}
\newcommand{\Z}{\mathbb{Z}}
\newcommand{\gU}{\mathsf{U}}
\newcommand{\fb}{f_{\mathrm{b}}}
\newcommand{\fbo}{f_{\mathrm{b}}^{\mathrm{o}}}
\newcommand{\kappat}{\tilde{\kappa}}
\newcommand{\EE}{\mathbb{E}}
\DeclareMathOperator{\onev}{{\mathbf{1}}}
\newcommand{\imagi}{\boldsymbol{\imath}}
\newcommand{\HH}{\mathbb{H}}
\newcommand{\sH}{\mathcal{H}}
\newcommand{\sE}{\mathcal{E}}
\newcommand{\sI}{\mathcal{I}}
\newcommand{\sT}{\mathcal{T}}
\newcommand{\aslone}{\text{a.s. and in } L^1 \text{ as } n \to \infty}
\newcommand{\blambda}{\boldsymbol{\lambda}}
\newcommand{\balpha}{\boldsymbol{\alpha}}

\DeclareMathOperator{\dist}{dist}
\DeclareMathOperator{\poisson}{Po}
\DeclareMathOperator{\tpoisson}{\widetilde{Po}}
\DeclareMathOperator{\Exp}{Exp}
\DeclareMathOperator{\gibbs}{Gibbs}
\DeclareMathOperator{\supp}{supp}
\DeclareMathOperator*{\esssup}{ess\,sup}
\newcommand{\eqinlaw}{\text{\raisebox{0pt}[0pt][0pt]{${}\stackrel{\mathscr{D}}{=}{}$}}}
\newcommand{\dst}{\displaystlye}
\newcommand{\tst}{\textstyle}
\newcommand{\sst}{\scriptstyle}
\newcommand{\ssst}{\scriptscriptstyle}
\newcommand{\tsum}{{\tst \sum}}
\newcommand{\dd}{\mathrm{d}}
\newcommand{\bs}{\boldsymbol}
\newcommand{\BB}{\mathbb{B}}
\newcommand{\NN}{\mathbb{N}}
\newcommand{\PP}{\mathbb{P}}
\newcommand{\RR}{\mathbb{R}}
\newcommand{\Rplus}{\mathbb{R}_{+}}
\newcommand{\Zplus}{\mathbb{Z}_{+}}
\newcommand{\mcf}{\mathcal{F}}
\newcommand{\mcx}{\mathcal{X}}
\newcommand{\dtv}{d_{{\mathrm TV}}}
\newcommand{\dzero}{d}
\newcommand{\done}{d_1}
\newcommand{\dtwo}{d_2}
\newcommand{\dtwobar}{\overline{\hspace*{-1.5pt}d\hspace*{1.5pt}}\hspace*{-1.5pt}_2}
\newcommand{\msa}{\mathscr{A}}
\newcommand{\msd}{\mathscr{D}}
\newcommand{\msl}{\mathscr{L}}
\newcommand{\bphi}{\breve{\phi}}
\newcommand{\bara}{\bar{a}}
\newcommand{\barb}{\bar{b}}
\newcommand{\bard}{\bar{d}}
\newcommand{\barlam}{\bar{\lambda}}
\newcommand{\abs}[1]{\lvert #1 \rvert}
\newcommand{\bigabs}[1]{\bigl| #1 \bigr|}
\newcommand{\Bigabs}[1]{\Bigl| #1 \Bigr|}
\newcommand{\biggabs}[1]{\biggl| #1 \biggr|}
\newcommand{\Biggabs}[1]{\Biggl| #1 \Biggr|}
\newcommand{\norm}[1]{\lVert #1 \rVert}
\newcommand{\bignorm}[1]{\bigl\lVert #1 \bigr\rVert}
\newcommand{\Bignorm}[1]{\Bigl\lVert #1 \Bigr\rVert}
\newcommand{\mca}{\mathcal{A}}
\newcommand{\mcb}{\mathcal{B}}
\newcommand{\mcn}{\mathcal{N}}
\newcommand{\mcw}{\mathcal{W}}
\newcommand{\mfn}{\mathfrak{N}}
\newcommand{\tf}{\tilde{f}}
\newcommand{\tilh}{\tilde{h}}
\newcommand{\tu}{\tilde{u}}
\newcommand{\tbalpha}{\widetilde{\balpha}}
\newcommand{\tlambda}{\tilde{\lambda}}
\newcommand{\tnu}{\tilde{\nu}}
\newcommand{\txi}{\tilde{\xi}}
\newcommand{\tXXi}{\widetilde{\XXi}}
\newcommand{\tEta}{\widetilde{\Eta}}
\newcommand{\tmcx}{\widetilde{\mcx}}
\newcommand{\ftv}{\mathcal{F}_{TV}}
\newcommand{\fone}{\mathcal{F}_{1}}
\newcommand{\ftwo}{\mathcal{F}_{2}}
\newcommand{\XXi}{\Xi}
\newcommand{\ssXXi}{\sst {\Xi \hspace*{-0.25em}\rule{0.02em}{1.0ex}\hspace*{0.22em}}}
\newcommand{\Eta}{\mathrm{H}}
\newcommand{\mvert}{\, \vert \,}
\newcommand{\one}{\mathbbm{1}}
\newcommand{\RD}{\RR^D}
\newcommand{\hbit}{\hspace*{1.5pt}}
\newcommand{\nhbit}{\hspace*{-1.5pt}}
\newcommand{\Leb}{\mathrm{Leb}}

\newcommand*{\bi}{\textbf{S}}
\newcommand*{\SSS}{{\upshape{\textbf{(S)}}}\xspace}
\newcommand*{\RS}{{\upshape{\textbf{(RS)}}}\xspace}
\newcommand*{\UB}{{\upshape{\textbf{(UB)}}}\xspace}
\newcommand*{\RC}{{\upshape{\textbf{(RC)}}}\xspace}
\newcommand*{\IR}{{\upshape{\textbf{(IR)}}}\xspace}
\DeclareMathOperator{\pip}{PIP}

\numberwithin{equation}{section}
\numberwithin{definition}{section}

\newlength{\querylen}
\usepackage{fancybox}
\newcommand{\query}[1]{\medskip \noindent
  \shadowbox{\begin{minipage}[t]{\querylen} {#1}
    \end{minipage}}\medskip}

\begin{document}

\title{Bounds for the probability generating functional of a Gibbs point process\footnote{Supported by Swiss
    National Science Foundation Grant 200021-137527.}}

\author{Kaspar Stucki\footnote{Institute of Mathematical Statistics and Actuarial Science, University of Bern, Alpeneggstrasse 22, 3012 Bern, Switzerland.}\ \footnote{Email address: kaspar.stucki@stat.unibe.ch} \hspace*{0pt} and Dominic Schuhmacher\footnotemark[2]\ \footnote{Current address: Institute for Mathematical Stochastics, University of G\"ottingen, Goldschmidtstra{\ss}e 7, 37077 G\"ottingen, Germany}\ \footnote{Email address: dominic.schuhmacher@mathematik.uni-goettingen.de}\\
University of Bern}

\date{18 April 2013}
\maketitle

\begin{abstract}
We derive explicit lower and upper bounds for the probability generating functional of a stationary locally stable Gibbs point process, which can be applied to summary statistics like the $F$ function. For pairwise interaction processes we obtain further estimates for the $G$ and $K$ functions, the intensity and higher order correlation functions.

The proof of the main result is based on Stein's method for Poisson point process approximation.  
\vspace{.3 cm}


\vspace{.3 cm}

\end{abstract}

\section{Introduction}
\label{sec:introduction}

Gibbs processes are very popular point process models which are extensively used both in spatial statistics and in statistical physics, see e.g.\ \citep{moellerwaage04,ruelle69}. Especially the pairwise interaction processes allow a simple yet flexible modelling of point interactions. However, a major drawback of Gibbs processes is that in general there are no analytic formulas available for their intensities or higher order correlation functions. In a recent couple of articles \citep{baddeleynair12, bn12} Baddeley and Nair proposed an approximation method that is fast to compute and accurate as verified by Monte Carlo methods. There are however no theoretical results in this respect and hence no guarantees for accuracy in most concrete models nor quantifications of the approximation error.

The aim of the present paper is to derive rigorous lower and upper bounds for correlation functions and related quantities. These allow us to narrow down the true values quite precisely if the Gibbs process is not too far away from a Poisson process. Figure~\ref{fig:Strauss50} shows our bounds on the intensity for a two-dimensional Strauss process in dependence of its interaction parameter~$\gamma$. The pluses are estimates of the true intensity obtained as averages over the numbers of points in $[0,1]^2$ of 10,000 Strauss processes simulated by dominated coupling from the past. The point processes were simulated on a larger window ($[-0.5,1.5]^2$) in order avoid noticeable edge effects. All simulations and numerical computations in this paper were performed in the R language \citep{r12} using the contributed package \texttt{spatstat} \citep{spatstat12}.

\begin{figure}[ht]
\begin{center}
\begin{tikzpicture}[x=1pt,y=1pt]
\definecolor[named]{fillColor}{rgb}{1.00,1.00,1.00}
\path[use as bounding box,fill=fillColor,fill opacity=0.00] (0,0) rectangle (254.39,231.26);
\begin{scope}
\path[clip] (  0.00,  0.00) rectangle (254.39,231.26);
\definecolor[named]{drawColor}{rgb}{0.00,0.00,0.00}

\path[draw=drawColor,line width= 0.4pt,line join=round,line cap=round] ( 50.11, 28.91) -- (218.74, 28.91);

\path[draw=drawColor,line width= 0.4pt,line join=round,line cap=round] ( 50.11, 28.91) -- ( 50.11, 22.91);

\path[draw=drawColor,line width= 0.4pt,line join=round,line cap=round] ( 83.83, 28.91) -- ( 83.83, 22.91);

\path[draw=drawColor,line width= 0.4pt,line join=round,line cap=round] (117.56, 28.91) -- (117.56, 22.91);

\path[draw=drawColor,line width= 0.4pt,line join=round,line cap=round] (151.29, 28.91) -- (151.29, 22.91);

\path[draw=drawColor,line width= 0.4pt,line join=round,line cap=round] (185.01, 28.91) -- (185.01, 22.91);

\path[draw=drawColor,line width= 0.4pt,line join=round,line cap=round] (218.74, 28.91) -- (218.74, 22.91);

\node[text=drawColor,anchor=base,inner sep=0pt, outer sep=0pt, scale=  0.70] at ( 50.11, 15.71) {0.0};

\node[text=drawColor,anchor=base,inner sep=0pt, outer sep=0pt, scale=  0.70] at ( 83.83, 15.71) {0.2};

\node[text=drawColor,anchor=base,inner sep=0pt, outer sep=0pt, scale=  0.70] at (117.56, 15.71) {0.4};

\node[text=drawColor,anchor=base,inner sep=0pt, outer sep=0pt, scale=  0.70] at (151.29, 15.71) {0.6};

\node[text=drawColor,anchor=base,inner sep=0pt, outer sep=0pt, scale=  0.70] at (185.01, 15.71) {0.8};

\node[text=drawColor,anchor=base,inner sep=0pt, outer sep=0pt, scale=  0.70] at (218.74, 15.71) {1.0};

\path[draw=drawColor,line width= 0.4pt,line join=round,line cap=round] ( 43.36, 35.87) -- ( 43.36,209.85);

\path[draw=drawColor,line width= 0.4pt,line join=round,line cap=round] ( 43.36, 35.87) -- ( 37.36, 35.87);

\path[draw=drawColor,line width= 0.4pt,line join=round,line cap=round] ( 43.36, 64.86) -- ( 37.36, 64.86);

\path[draw=drawColor,line width= 0.4pt,line join=round,line cap=round] ( 43.36, 93.86) -- ( 37.36, 93.86);

\path[draw=drawColor,line width= 0.4pt,line join=round,line cap=round] ( 43.36,122.86) -- ( 37.36,122.86);

\path[draw=drawColor,line width= 0.4pt,line join=round,line cap=round] ( 43.36,151.86) -- ( 37.36,151.86);

\path[draw=drawColor,line width= 0.4pt,line join=round,line cap=round] ( 43.36,180.85) -- ( 37.36,180.85);

\path[draw=drawColor,line width= 0.4pt,line join=round,line cap=round] ( 43.36,209.85) -- ( 37.36,209.85);

\node[text=drawColor,anchor=base east,inner sep=0pt, outer sep=0pt, scale=  0.70] at ( 34.96, 33.46) {20};

\node[text=drawColor,anchor=base east,inner sep=0pt, outer sep=0pt, scale=  0.70] at ( 34.96, 62.45) {25};

\node[text=drawColor,anchor=base east,inner sep=0pt, outer sep=0pt, scale=  0.70] at ( 34.96, 91.45) {30};

\node[text=drawColor,anchor=base east,inner sep=0pt, outer sep=0pt, scale=  0.70] at ( 34.96,120.45) {35};

\node[text=drawColor,anchor=base east,inner sep=0pt, outer sep=0pt, scale=  0.70] at ( 34.96,149.45) {40};

\node[text=drawColor,anchor=base east,inner sep=0pt, outer sep=0pt, scale=  0.70] at ( 34.96,178.44) {45};

\node[text=drawColor,anchor=base east,inner sep=0pt, outer sep=0pt, scale=  0.70] at ( 34.96,207.44) {50};
\end{scope}
\begin{scope}
\path[clip] ( 43.36, 28.91) rectangle (225.48,216.81);
\definecolor[named]{fillColor}{rgb}{0.83,0.83,0.83}

\path[fill=fillColor] ( 50.11,138.76) --
	( 51.79,139.20) --
	( 53.48,139.65) --
	( 55.17,140.09) --
	( 56.85,140.54) --
	( 58.54,141.00) --
	( 60.23,141.46) --
	( 61.91,141.92) --
	( 63.60,142.38) --
	( 65.28,142.85) --
	( 66.97,143.33) --
	( 68.66,143.80) --
	( 70.34,144.28) --
	( 72.03,144.77) --
	( 73.72,145.26) --
	( 75.40,145.75) --
	( 77.09,146.25) --
	( 78.77,146.75) --
	( 80.46,147.25) --
	( 82.15,147.76) --
	( 83.83,148.28) --
	( 85.52,148.79) --
	( 87.21,149.32) --
	( 88.89,149.84) --
	( 90.58,150.37) --
	( 92.26,150.91) --
	( 93.95,151.45) --
	( 95.64,152.00) --
	( 97.32,152.55) --
	( 99.01,153.10) --
	(100.70,153.66) --
	(102.38,154.23) --
	(104.07,154.80) --
	(105.76,155.37) --
	(107.44,155.95) --
	(109.13,156.54) --
	(110.81,157.13) --
	(112.50,157.72) --
	(114.19,158.32) --
	(115.87,158.93) --
	(117.56,159.54) --
	(119.25,160.16) --
	(120.93,160.78) --
	(122.62,161.41) --
	(124.30,162.04) --
	(125.99,162.68) --
	(127.68,163.33) --
	(129.36,163.98) --
	(131.05,164.64) --
	(132.74,165.31) --
	(134.42,165.98) --
	(136.11,166.65) --
	(137.79,167.34) --
	(139.48,168.03) --
	(141.17,168.72) --
	(142.85,169.43) --
	(144.54,170.14) --
	(146.23,170.86) --
	(147.91,171.58) --
	(149.60,172.31) --
	(151.29,173.05) --
	(152.97,173.79) --
	(154.66,174.55) --
	(156.34,175.31) --
	(158.03,176.08) --
	(159.72,176.85) --
	(161.40,177.63) --
	(163.09,178.43) --
	(164.78,179.22) --
	(166.46,180.03) --
	(168.15,180.85) --
	(169.83,181.67) --
	(171.52,182.50) --
	(173.21,183.35) --
	(174.89,184.20) --
	(176.58,185.05) --
	(178.27,185.92) --
	(179.95,186.80) --
	(181.64,187.69) --
	(183.32,188.58) --
	(185.01,189.49) --
	(186.70,190.40) --
	(188.38,191.33) --
	(190.07,192.26) --
	(191.76,193.21) --
	(193.44,194.16) --
	(195.13,195.13) --
	(196.82,196.10) --
	(198.50,197.09) --
	(200.19,198.09) --
	(201.87,199.10) --
	(203.56,200.12) --
	(205.25,201.15) --
	(206.93,202.20) --
	(208.62,203.25) --
	(210.31,204.32) --
	(211.99,205.40) --
	(213.68,206.49) --
	(215.36,207.60) --
	(217.05,208.72) --
	(218.74,209.85) --
	(218.74,209.85) --
	(217.05,208.72) --
	(215.36,207.59) --
	(213.68,206.47) --
	(211.99,205.37) --
	(210.31,204.27) --
	(208.62,203.18) --
	(206.93,202.09) --
	(205.25,201.02) --
	(203.56,199.95) --
	(201.87,198.89) --
	(200.19,197.84) --
	(198.50,196.80) --
	(196.82,195.77) --
	(195.13,194.74) --
	(193.44,193.72) --
	(191.76,192.71) --
	(190.07,191.70) --
	(188.38,190.71) --
	(186.70,189.72) --
	(185.01,188.73) --
	(183.32,187.76) --
	(181.64,186.79) --
	(179.95,185.83) --
	(178.27,184.88) --
	(176.58,183.93) --
	(174.89,182.99) --
	(173.21,182.05) --
	(171.52,181.13) --
	(169.83,180.20) --
	(168.15,179.29) --
	(166.46,178.38) --
	(164.78,177.48) --
	(163.09,176.58) --
	(161.40,175.69) --
	(159.72,174.81) --
	(158.03,173.93) --
	(156.34,173.06) --
	(154.66,172.20) --
	(152.97,171.34) --
	(151.29,170.49) --
	(149.60,169.64) --
	(147.91,168.80) --
	(146.23,167.96) --
	(144.54,167.13) --
	(142.85,166.30) --
	(141.17,165.48) --
	(139.48,164.67) --
	(137.79,163.86) --
	(136.11,163.06) --
	(134.42,162.26) --
	(132.74,161.47) --
	(131.05,160.68) --
	(129.36,159.90) --
	(127.68,159.12) --
	(125.99,158.35) --
	(124.30,157.58) --
	(122.62,156.82) --
	(120.93,156.06) --
	(119.25,155.30) --
	(117.56,154.56) --
	(115.87,153.81) --
	(114.19,153.07) --
	(112.50,152.34) --
	(110.81,151.61) --
	(109.13,150.89) --
	(107.44,150.16) --
	(105.76,149.45) --
	(104.07,148.74) --
	(102.38,148.03) --
	(100.70,147.33) --
	( 99.01,146.63) --
	( 97.32,145.93) --
	( 95.64,145.24) --
	( 93.95,144.56) --
	( 92.26,143.88) --
	( 90.58,143.20) --
	( 88.89,142.53) --
	( 87.21,141.86) --
	( 85.52,141.19) --
	( 83.83,140.53) --
	( 82.15,139.87) --
	( 80.46,139.22) --
	( 78.77,138.57) --
	( 77.09,137.92) --
	( 75.40,137.28) --
	( 73.72,136.64) --
	( 72.03,136.01) --
	( 70.34,135.38) --
	( 68.66,134.75) --
	( 66.97,134.13) --
	( 65.28,133.51) --
	( 63.60,132.89) --
	( 61.91,132.28) --
	( 60.23,131.67) --
	( 58.54,131.06) --
	( 56.85,130.46) --
	( 55.17,129.86) --
	( 53.48,129.27) --
	( 51.79,128.68) --
	( 50.11,128.09) --
	cycle;
\definecolor[named]{drawColor}{rgb}{0.00,0.00,0.00}

\path[draw=drawColor,line width= 0.2pt,line join=round,line cap=round] ( 50.11,138.76) --
	( 51.79,139.20) --
	( 53.48,139.65) --
	( 55.17,140.09) --
	( 56.85,140.54) --
	( 58.54,141.00) --
	( 60.23,141.46) --
	( 61.91,141.92) --
	( 63.60,142.38) --
	( 65.28,142.85) --
	( 66.97,143.33) --
	( 68.66,143.80) --
	( 70.34,144.28) --
	( 72.03,144.77) --
	( 73.72,145.26) --
	( 75.40,145.75) --
	( 77.09,146.25) --
	( 78.77,146.75) --
	( 80.46,147.25) --
	( 82.15,147.76) --
	( 83.83,148.28) --
	( 85.52,148.79) --
	( 87.21,149.32) --
	( 88.89,149.84) --
	( 90.58,150.37) --
	( 92.26,150.91) --
	( 93.95,151.45) --
	( 95.64,152.00) --
	( 97.32,152.55) --
	( 99.01,153.10) --
	(100.70,153.66) --
	(102.38,154.23) --
	(104.07,154.80) --
	(105.76,155.37) --
	(107.44,155.95) --
	(109.13,156.54) --
	(110.81,157.13) --
	(112.50,157.72) --
	(114.19,158.32) --
	(115.87,158.93) --
	(117.56,159.54) --
	(119.25,160.16) --
	(120.93,160.78) --
	(122.62,161.41) --
	(124.30,162.04) --
	(125.99,162.68) --
	(127.68,163.33) --
	(129.36,163.98) --
	(131.05,164.64) --
	(132.74,165.31) --
	(134.42,165.98) --
	(136.11,166.65) --
	(137.79,167.34) --
	(139.48,168.03) --
	(141.17,168.72) --
	(142.85,169.43) --
	(144.54,170.14) --
	(146.23,170.86) --
	(147.91,171.58) --
	(149.60,172.31) --
	(151.29,173.05) --
	(152.97,173.79) --
	(154.66,174.55) --
	(156.34,175.31) --
	(158.03,176.08) --
	(159.72,176.85) --
	(161.40,177.63) --
	(163.09,178.43) --
	(164.78,179.22) --
	(166.46,180.03) --
	(168.15,180.85) --
	(169.83,181.67) --
	(171.52,182.50) --
	(173.21,183.35) --
	(174.89,184.20) --
	(176.58,185.05) --
	(178.27,185.92) --
	(179.95,186.80) --
	(181.64,187.69) --
	(183.32,188.58) --
	(185.01,189.49) --
	(186.70,190.40) --
	(188.38,191.33) --
	(190.07,192.26) --
	(191.76,193.21) --
	(193.44,194.16) --
	(195.13,195.13) --
	(196.82,196.10) --
	(198.50,197.09) --
	(200.19,198.09) --
	(201.87,199.10) --
	(203.56,200.12) --
	(205.25,201.15) --
	(206.93,202.20) --
	(208.62,203.25) --
	(210.31,204.32) --
	(211.99,205.40) --
	(213.68,206.49) --
	(215.36,207.60) --
	(217.05,208.72) --
	(218.74,209.85);

\path[draw=drawColor,line width= 0.2pt,line join=round,line cap=round] ( 50.11,128.09) --
	( 51.79,128.68) --
	( 53.48,129.27) --
	( 55.17,129.86) --
	( 56.85,130.46) --
	( 58.54,131.06) --
	( 60.23,131.67) --
	( 61.91,132.28) --
	( 63.60,132.89) --
	( 65.28,133.51) --
	( 66.97,134.13) --
	( 68.66,134.75) --
	( 70.34,135.38) --
	( 72.03,136.01) --
	( 73.72,136.64) --
	( 75.40,137.28) --
	( 77.09,137.92) --
	( 78.77,138.57) --
	( 80.46,139.22) --
	( 82.15,139.87) --
	( 83.83,140.53) --
	( 85.52,141.19) --
	( 87.21,141.86) --
	( 88.89,142.53) --
	( 90.58,143.20) --
	( 92.26,143.88) --
	( 93.95,144.56) --
	( 95.64,145.24) --
	( 97.32,145.93) --
	( 99.01,146.63) --
	(100.70,147.33) --
	(102.38,148.03) --
	(104.07,148.74) --
	(105.76,149.45) --
	(107.44,150.16) --
	(109.13,150.89) --
	(110.81,151.61) --
	(112.50,152.34) --
	(114.19,153.07) --
	(115.87,153.81) --
	(117.56,154.56) --
	(119.25,155.30) --
	(120.93,156.06) --
	(122.62,156.82) --
	(124.30,157.58) --
	(125.99,158.35) --
	(127.68,159.12) --
	(129.36,159.90) --
	(131.05,160.68) --
	(132.74,161.47) --
	(134.42,162.26) --
	(136.11,163.06) --
	(137.79,163.86) --
	(139.48,164.67) --
	(141.17,165.48) --
	(142.85,166.30) --
	(144.54,167.13) --
	(146.23,167.96) --
	(147.91,168.80) --
	(149.60,169.64) --
	(151.29,170.49) --
	(152.97,171.34) --
	(154.66,172.20) --
	(156.34,173.06) --
	(158.03,173.93) --
	(159.72,174.81) --
	(161.40,175.69) --
	(163.09,176.58) --
	(164.78,177.48) --
	(166.46,178.38) --
	(168.15,179.29) --
	(169.83,180.20) --
	(171.52,181.13) --
	(173.21,182.05) --
	(174.89,182.99) --
	(176.58,183.93) --
	(178.27,184.88) --
	(179.95,185.83) --
	(181.64,186.79) --
	(183.32,187.76) --
	(185.01,188.73) --
	(186.70,189.72) --
	(188.38,190.71) --
	(190.07,191.70) --
	(191.76,192.71) --
	(193.44,193.72) --
	(195.13,194.74) --
	(196.82,195.77) --
	(198.50,196.80) --
	(200.19,197.84) --
	(201.87,198.89) --
	(203.56,199.95) --
	(205.25,201.02) --
	(206.93,202.09) --
	(208.62,203.18) --
	(210.31,204.27) --
	(211.99,205.37) --
	(213.68,206.47) --
	(215.36,207.59) --
	(217.05,208.72) --
	(218.74,209.85);

\node[text=drawColor,anchor=base,inner sep=0pt, outer sep=0pt, scale=  0.70] at ( 50.11,129.88) {+};

\node[text=drawColor,anchor=base,inner sep=0pt, outer sep=0pt, scale=  0.70] at ( 58.54,132.55) {+};

\node[text=drawColor,anchor=base,inner sep=0pt, outer sep=0pt, scale=  0.70] at ( 66.97,136.21) {+};

\node[text=drawColor,anchor=base,inner sep=0pt, outer sep=0pt, scale=  0.70] at ( 75.40,139.52) {+};

\node[text=drawColor,anchor=base,inner sep=0pt, outer sep=0pt, scale=  0.70] at ( 83.83,142.40) {+};

\node[text=drawColor,anchor=base,inner sep=0pt, outer sep=0pt, scale=  0.70] at ( 92.26,146.00) {+};

\node[text=drawColor,anchor=base,inner sep=0pt, outer sep=0pt, scale=  0.70] at (100.70,148.24) {+};

\node[text=drawColor,anchor=base,inner sep=0pt, outer sep=0pt, scale=  0.70] at (109.13,152.28) {+};

\node[text=drawColor,anchor=base,inner sep=0pt, outer sep=0pt, scale=  0.70] at (117.56,155.59) {+};

\node[text=drawColor,anchor=base,inner sep=0pt, outer sep=0pt, scale=  0.70] at (125.99,159.26) {+};

\node[text=drawColor,anchor=base,inner sep=0pt, outer sep=0pt, scale=  0.70] at (134.42,162.66) {+};

\node[text=drawColor,anchor=base,inner sep=0pt, outer sep=0pt, scale=  0.70] at (142.85,166.11) {+};

\node[text=drawColor,anchor=base,inner sep=0pt, outer sep=0pt, scale=  0.70] at (151.29,170.42) {+};

\node[text=drawColor,anchor=base,inner sep=0pt, outer sep=0pt, scale=  0.70] at (159.72,174.24) {+};

\node[text=drawColor,anchor=base,inner sep=0pt, outer sep=0pt, scale=  0.70] at (168.15,178.79) {+};

\node[text=drawColor,anchor=base,inner sep=0pt, outer sep=0pt, scale=  0.70] at (176.58,182.46) {+};

\node[text=drawColor,anchor=base,inner sep=0pt, outer sep=0pt, scale=  0.70] at (185.01,187.50) {+};

\node[text=drawColor,anchor=base,inner sep=0pt, outer sep=0pt, scale=  0.70] at (193.44,192.28) {+};

\node[text=drawColor,anchor=base,inner sep=0pt, outer sep=0pt, scale=  0.70] at (201.87,196.71) {+};

\node[text=drawColor,anchor=base,inner sep=0pt, outer sep=0pt, scale=  0.70] at (210.31,203.13) {+};

\node[text=drawColor,anchor=base,inner sep=0pt, outer sep=0pt, scale=  0.70] at (218.74,208.13) {+};
\end{scope}
\begin{scope}
\path[clip] (  0.00,  0.00) rectangle (254.39,231.26);
\definecolor[named]{drawColor}{rgb}{0.00,0.00,0.00}

\path[draw=drawColor,line width= 0.4pt,line join=round,line cap=round] ( 43.36, 28.91) --
	(225.48, 28.91) --
	(225.48,216.81) --
	( 43.36,216.81) --
	( 43.36, 28.91);

\node[text=drawColor,rotate= 90.00,anchor=base,inner sep=0pt, outer sep=0pt, scale=  0.70] at ( 16.96,122.86) {$\lambda$};

\node[text=drawColor,anchor=base,inner sep=0pt, outer sep=0pt, scale=  0.70] at (134.42,  7.31) {$\gamma$};
\end{scope}
\end{tikzpicture}
\caption{\label{fig:Strauss50} Bounds on the intensities of two-dimensional Strauss processes with $\beta=50$, $r=0.05$ and values of $\gamma$ ranging from $0$ to $1$. The pluses are estimates of the intensities based on 10,000 simulations each.}
\end{center}
\end{figure}

Our main result, Theorem~\ref{thm:bounds}, more generally gives bounds on the probability generating functional of a Gibbs process. Let $\Xi$ be an arbitrary point process on $\R^d$. The \emph{probability generating functional (p.g.fl.)} $\Psi_{\Xi}$ is defined as 
\begin{equation}
\label{eq:pgfl}
\Psi_\Xi(g)=\E \Bigl(\prod_{y\in \Xi}g(y)\Bigr)
\end{equation} 
for any measurable function $g\colon \R^d\to [0,1]$ for which $1-g$ has bounded support,
see e.g.\ \cite[p.\,59]{dvj08} for details.

Many statistics of point processes, such as the \emph{empty space function} ($F$ function), contain expectations as in \eqref{eq:pgfl}. For pairwise interaction processes the situation is even better. By the Georgii--Nguyen--Zessin equation~\eqref{eq:gnz} the \emph{nearest neighbour function} ($G$ function), Ripley's $K$ function and the correlation functions of all orders can be rewritten using the p.g.fl.

The idea for proving Theorem~\ref{thm:bounds} is to replace the Gibbs process $\Xi$ in \eqref{eq:pgfl} by a suitable Poisson process and bound the error using Stein's method.

The rest of the paper is organised as follows. In Section~\ref{sec:pre} we introduce some notation and state the main result. In Section~\ref{sec:bounds-int} we provide bounds on the intensity, and in Section~\ref{sec:s-stat}  bounds on other summary statistics are derived. Section~\ref{sec:proofs} contains the proof of the main result.

\section{Preliminaries and main result}
\label{sec:pre}

Let $(\mfn,\mcn)$ denote the space of locally finite point measures on $\R^d$ equipped with the $\sigma$-algebra generated by the evaluation maps $[\mfn\ni\xi\mapsto \xi(A) \in \Zplus]$ for bounded Borel sets $A\in \R^d$. A point process is just a $\mfn$-valued random element. We assume the point processes to be \emph{simple}, i.e.\ do not allow multi-points. Thus we can use set notation, i.e.\ $x\in \xi$ means that the point $x$ lies in the support of the measure $\xi$.
 
In spatial statistics point processes are usually defined on a bounded window $\mcw\subset \R^d$. Let $\mfn\vert_\mcw$ denote the restriction of the $\mfn$ to $\mcw$. A point process $\Xi$ on $\mcw$ is called a \emph{Gibbs process} if it has a hereditary density $u$ with respect to the distribution of the Poisson process with unit intensity. Hereditarity means that $u(\xi)>0$ implies $u(\eta)>0$ for all subconfigurations 
$\eta\subset \xi$. Invoking the Hammersley--Clifford--Ripley--Kelly theorem in \cite[Thm.~6.1]{moellerwaage04}, such a process may be characterised by requiring a density of the form $u(\xi) = \exp \bigl(-\sum_{n=1}^\infty \sum_{\{x_1,\ldots,x_n\} \subset \xi} v(x_1,\ldots,x_n) \bigr)$, where $v \colon \bigl( \R^d \bigr)^n \to (-\infty, \infty]$ are symmetric functions, which yields the definition of a finite Gibbs process from statistical physics. By hereditarity we can define the \emph{conditional intensity} as
\begin{equation}
\label{eq:cond-int}
\lambda(x\mid \xi)=\frac{u(\xi\cup\{x\})}{u(\xi)},
\end{equation}
where $0/0=0$. Roughly speaking, the conditional intensity is the infinitesimal probability that $\Xi$ has a point at $x$, given that $\Xi$ coincides with the configuration $\xi$ everywhere else. Furthermore $\lambda(\cdot\mvert \cdot)$ uniquely characterises the distribution of $\Xi$, since by \eqref{eq:cond-int} one can recursively recover an unnormalised density. It is well-known that the conditional intensity is the $dx \otimes \msl(\XXi)$-almost everywhere unique product measurable function that satisfies the \emph{Georgii--Nguyen--Zessin equation}, see \citep[p.~95]{moellerwaage04}
\begin{equation}
\label{eq:gnz}
  \EE \biggl( \int_{\mcw} h(x, \XXi\setminus\{x\}) \; \XXi(\dd x) \biggr) = \int_{\mcw} \EE \bigl( h(x, \XXi) \lambda(x \mvert \XXi) \bigr) \; dx
\end{equation}
for every measurable $h \colon \mcw \times \mfn\vert_\mcw \to \Rplus$.

So far $\lambda(\cdot\mvert\cdot)$ is only a function on $\mcw\times \mfn\vert_\mcw$, but in many cases there exists a natural extension to the whole space, which we shall also denote by $\lambda(\cdot\mvert\cdot)$.
One way to generalise Gibbs processes to the whole space $\R^d$ is then by the so-called \emph{integral characterisation}. A point process $\Xi$ on $\R^d$ is a Gibbs process corresponding to the conditional intensity $\lambda(\cdot\mvert \cdot)$ if it satisfies \eqref{eq:gnz} with $\mcw$ replaced by $\RR^d$ for all measurable $h \colon \R^d\times \mfn \to \Rplus$; see \citep[p.\,95]{moellerwaage04}, or \citep{nguyenzessin79} for a more rigorous presentation. Unlike in the case of a bounded domain, $\Xi$ may not be uniquely determined by~\eqref{eq:gnz}. For the rest of this paper we will only deal with the conditional intensity, i.e.\ if we say that a result holds for a Gibbs process with conditional intensity $\lambda(\cdot\mvert \cdot)$, we mean that it holds for \emph{all} processes corresponding to this conditional intensity.

A Gibbs process $\Xi$ is said to be a \emph{pairwise interaction process} if its conditional intensity is of the form
\begin{equation}
\label{eq:pip}
\lambda(x\mvert \xi)=\beta \prod_{y \in \xi}\varphi (x,y)
\end{equation}
for a constant $\beta>0$ and a symmetric \emph{interaction function} $\varphi$. We denote the distribution of $\Xi$ by $\pip(\beta,\varphi)$. The process $\Xi$ is called \emph{inhibitory} if $\varphi \le 1$ and it is said to have a \emph{finite interaction range} if $1-\varphi$ is compactly supported. $\Xi$ is \emph{stationary} if its distribution is invariant under translations, and \emph{isotropic} if it is also invariant under rotations. If $\Xi$ is stationary we tacitly assume that $\varphi(x,y)$ depends only on the difference $x-y$; we then write $\varphi(x,y)=\varphi(x-y)$. For conditions on $\varphi$ ensuring the existence of $\Xi$ we refer the reader to~\citep{ruelle69}.

If $\Xi$ is a general point process, its \emph{expectation measure or first order moment measure} $\E \Xi$ on $\RR^d$ is simply given by $(\E\Xi)(A) = \E(\Xi(A))$ for every Borel set $A \subset \RR^d$. For $k \geq 1$ the \emph{$k$-th order factorial moment measure} of $\Xi$ is the expectation measure of the factorial product measure
\begin{equation*}
  \Xi^{[k]} = \sum_{\substack{X_1,\ldots,X_k \in \Xi \\[2pt] \text{pairwise different}}} \delta_{(X_1,\ldots,X_k)}
\end{equation*}
on $(\RR^d)^k$. Any moment measure is said to \emph{exist} if it is locally finite.

The \emph{intensity (function)} $\lambda(x)$ of a Gibbs process $\Xi$ is the density of the first moment measure of $\Xi$ with respect to Lebesgue measure, provided the first moment measure exists. For a bounded $A\subset \R^d$, Equation~\eqref{eq:gnz} yields
\begin{equation*}
\E \Xi(A)=\int_A \E\big(\lambda(x\mvert \Xi)\big)\; dx,
\end{equation*} 
hence the existence of the intensity and the form $\lambda(x)=\E(\lambda(x\mvert \Xi))$. For stationary processes the intensity is obviously constant and we just write $\lambda$. For a stationary pairwise interaction process we get 
\begin{equation}
\label{eq:lambda-pip}
\lambda=\E\big(\lambda(0\mvert \Xi)\big)=\beta \hbit \E\Bigl(\prod_{y\in \Xi}\varphi(y)\Bigr)=\beta \hbit \Psi_\Xi( \varphi).
\end{equation}

In a similar manner it is possible to obtain the densities of the higher order factorial moment measures, the so-called \emph{correlation functions}; see \citep{mw07,mase90}. For a stationary process $\Xi\sim\pip(\beta,\varphi)$ the $k$-th correlation function is given by
\begin{align}
\lambda_k(x_1,\dots,x_k)&=\beta^k\Bigl(\prod_{1\le i<j\le k}\varphi(x_i-x_j)\Bigr)\hbit \E\Bigl(\prod_{y\in \Xi}\varphi(y-x_1)\cdots\varphi(y-x_k)\Bigr) \nonumber \\
&=\beta^k\Bigl(\prod_{1\le i<j\le k}\varphi(x_i-x_j)\Bigr)\hbit \Psi_\Xi\big(\varphi(\cdot-x_1)\cdots\varphi(\cdot-x_k)\big).
\label{eq:corr-fun}
\end{align}
A frequently used function in spatial statistics is the \emph{pair correlation function} which is defined as
\begin{equation}
\label{eq:pcf}
\rho(x,y)=\frac{\lambda_2(x,y)}{\lambda(x)\lambda(y)}.
\end{equation}
In the stationary isotropic case this simplifies to
$\rho(s)=\lambda_2(x,y)/\lambda^2$, where $s=\|x-y\|$.

For our results we need a stability condition for the Gibbs processes. A Gibbs process $\Xi$ is called \emph{locally stable} if there exists a non-negative function $c^*$ such that $\int_\mcw c^*(x)\,dx < \infty$ for all bounded domains $\mcw\subset \R^d$ and the conditional intensity satisfies
\begin{equation}
\label{eq:loc-s}
\lambda(x\mvert \xi)\le c^*(x),
\end{equation}
for all $\xi \in \mfn$. For stationary Gibbs processes we require $c^*$ to be constant.
Local stability is a frequently used condition in spatial statistics, see \citep[p.~84]{moellerwaage04}. However for $\Xi\sim\pip(\beta,\varphi)$ it implies that either $\Xi$ is inhibitory or it has a \emph{hard core}, i.e.\ there exists an $r_{\min}$ such that $\varphi(x,y)=0$ for all $\abs{x-y}\le r_{\min}$. This excludes for example the \emph{Lennard--Jones process}, which is very popular in statistical phyisics.

The following is the key theorem for obtaining the results in Sections~\ref{sec:bounds-int} and~\ref{sec:s-stat}. Its proof is the subject of Section~\ref{sec:proofs}.
\begin{theorem}
\label{thm:bounds}
Let $\Xi$ be a stationary locally stable Gibbs process with intensity $\lambda$ and local stability constant $c^*$, and let $g\colon \R^d\to [0,1]$ be a function for which $1-g$ has bounded support. Then
\begin{equation}
\label{eq:bounds}
1-\lambda G \le \E\Bigl(\prod_{y\in \Xi}g(y)\Bigr) \le 1-\frac{\lambda}{c^*}\big(1-e^{-c^*G}\big),
\end{equation}
where $G=\int_{\R^d} 1-g(x)\, dx$.
\end{theorem}

\begin{remark}
With a slight adaptation of the proof of Theorem~\ref{thm:bounds} it is possible to obtain bounds on the p.g.fl.s of possibly non-stationary processes. Let $\Xi$ be a locally stable Gibbs process with intensity function $\lambda(x)$ and local stability function $c^*(x)$.
Let $g\colon \R^d\to [0,1]$ be a function for which $1-g$ has bounded support. Then
\begin{multline}
\label{eq:bounds-non-stat}
1-\int_{\R^d}(1-g(x))\lambda(x)\;dx \le \E\Bigl(\prod_{y\in \Xi}g(y)\Bigr) \\
\le 1-\frac{\int_{\R^d}(1-g(x))\lambda(x)\;dx}{\int_{\R^d}(1-g(x))c^*(x)\;dx}\left(1-\exp\Big(-\int_{\R^d}(1-g(x))c^*(x)\;dx\Big)\right).
\end{multline}
As the intensity appears under the integral sign, we cannot generalise the arguments of Section~\ref{sec:bounds-int} based on \eqref{eq:bounds-non-stat}.
\end{remark}

For the rest of the paper we tacitly assume all point processes to be stationary.

\section{Bounds on the intensity}
\label{sec:bounds-int}

For the intensity of a inhibitory pairwise interaction process we immediately obtain from Theorem~\ref{thm:bounds} the following result.

\begin{theorem}
\label{thm:lambdaPIP}
Let $\Xi\sim \pip(\beta,\varphi)$ be inhibitory with finite interaction range. Then
\begin{equation}
\label{eq:lambdaPIP}
\frac{\beta}{1+\beta G}\le \lambda \le \frac{\beta}{2-e^{-\beta G}},
\end{equation}
where $G=\int_{\R^d}1-\varphi(x)\,dx$.
\end{theorem}

\begin{proof}
Recall from \eqref{eq:lambda-pip} that $\lambda=\beta \hbit \E \prod_{y\in \Xi}\varphi(y)$ and use $c^*=\beta$. Theorem~\ref{thm:bounds} then yields
\begin{equation*}
1-\lambda G \le \frac{\lambda}{\beta}\le 1-\frac{\lambda}{\beta}\big(1-e^{-\beta G}\big)
\end{equation*}
which can be rearranged as \eqref{eq:lambdaPIP}.
\end{proof}

\begin{remark}
The lower bound of \eqref{eq:lambdaPIP} can also be found in \cite[p.\,96]{ruelle69} with the restriction $\beta G <e^{-1}$, whereas our inequality holds for all values of $\beta$ and $G$.
\end{remark}

\begin{example}
Let $\Xi$ be a Strauss process, i.e.
\begin{equation*}
\varphi(x) =\begin{cases} \gamma \quad &\text{if} \quad  \|x\| \le r\\
1 \quad &\text{if} \quad \|x\|>r \end{cases}
\end{equation*}
for some parameters $r>0$ and $0\le \gamma \le 1$. Then $G=(1-\gamma)\alpha_dr^d$, where $\alpha_d$ denotes the volume of the unit ball. Figure~\ref{fig:Strauss50} shows that for a reasonable choice of the parameters $(\beta,r,\gamma)$ the bounds on $\lambda$ are quite good. The maximal relative error between the bounds and the simulated values is about 3.5\%. 
\end{example}

\begin{remark}
For the special case of a Strauss hard core process on $\RR^2$, i.e.\ setting $\gamma=0$ in the above example, intensity estimates are derived in \cite[p.\,181]{skm95}. The procedure used there can easily be refined to yield the bounds
\begin{equation*}
  \frac{\beta}{1+\beta \pi r^2} \leq \lambda \leq \frac{\beta}{1+\beta \pi r^2/4}
\end{equation*}
(Dietrich Stoyan, 2012, personal communication). The lower bound is the same as implied by our Theorem~\ref{thm:lambdaPIP}. The upper bound is worse for small and moderate choices of $\beta$, but better if $\beta$ is very large (e.g.\ $\geq 500$ if $r=0.05$).
\end{remark}

\begin{figure}[!ht]
\begin{center}
\begin{tikzpicture}[x=1pt,y=1pt]
\definecolor[named]{fillColor}{rgb}{1.00,1.00,1.00}
\path[use as bounding box,fill=fillColor,fill opacity=0.00] (0,0) rectangle (254.39,231.26);
\begin{scope}
\path[clip] (  0.00,  0.00) rectangle (254.39,231.26);
\definecolor[named]{drawColor}{rgb}{0.00,0.00,0.00}

\path[draw=drawColor,line width= 0.4pt,line join=round,line cap=round] ( 50.64, 28.91) -- (232.66, 28.91);

\path[draw=drawColor,line width= 0.4pt,line join=round,line cap=round] ( 50.64, 28.91) -- ( 50.64, 22.91);

\path[draw=drawColor,line width= 0.4pt,line join=round,line cap=round] ( 87.05, 28.91) -- ( 87.05, 22.91);

\path[draw=drawColor,line width= 0.4pt,line join=round,line cap=round] (123.45, 28.91) -- (123.45, 22.91);

\path[draw=drawColor,line width= 0.4pt,line join=round,line cap=round] (159.85, 28.91) -- (159.85, 22.91);

\path[draw=drawColor,line width= 0.4pt,line join=round,line cap=round] (196.25, 28.91) -- (196.25, 22.91);

\path[draw=drawColor,line width= 0.4pt,line join=round,line cap=round] (232.66, 28.91) -- (232.66, 22.91);

\node[text=drawColor,anchor=base,inner sep=0pt, outer sep=0pt, scale=  0.70] at ( 50.64, 15.71) {0.0};

\node[text=drawColor,anchor=base,inner sep=0pt, outer sep=0pt, scale=  0.70] at ( 87.05, 15.71) {0.2};

\node[text=drawColor,anchor=base,inner sep=0pt, outer sep=0pt, scale=  0.70] at (123.45, 15.71) {0.4};

\node[text=drawColor,anchor=base,inner sep=0pt, outer sep=0pt, scale=  0.70] at (159.85, 15.71) {0.6};

\node[text=drawColor,anchor=base,inner sep=0pt, outer sep=0pt, scale=  0.70] at (196.25, 15.71) {0.8};

\node[text=drawColor,anchor=base,inner sep=0pt, outer sep=0pt, scale=  0.70] at (232.66, 15.71) {1.0};

\path[draw=drawColor,line width= 0.4pt,line join=round,line cap=round] ( 43.36, 35.87) -- ( 43.36,209.85);

\path[draw=drawColor,line width= 0.4pt,line join=round,line cap=round] ( 43.36, 35.87) -- ( 37.36, 35.87);

\path[draw=drawColor,line width= 0.4pt,line join=round,line cap=round] ( 43.36, 79.36) -- ( 37.36, 79.36);

\path[draw=drawColor,line width= 0.4pt,line join=round,line cap=round] ( 43.36,122.86) -- ( 37.36,122.86);

\path[draw=drawColor,line width= 0.4pt,line join=round,line cap=round] ( 43.36,166.35) -- ( 37.36,166.35);

\path[draw=drawColor,line width= 0.4pt,line join=round,line cap=round] ( 43.36,209.85) -- ( 37.36,209.85);

\node[text=drawColor,anchor=base east,inner sep=0pt, outer sep=0pt, scale=  0.70] at ( 34.96, 33.46) {20};

\node[text=drawColor,anchor=base east,inner sep=0pt, outer sep=0pt, scale=  0.70] at ( 34.96, 76.95) {40};

\node[text=drawColor,anchor=base east,inner sep=0pt, outer sep=0pt, scale=  0.70] at ( 34.96,120.45) {60};

\node[text=drawColor,anchor=base east,inner sep=0pt, outer sep=0pt, scale=  0.70] at ( 34.96,163.94) {80};

\node[text=drawColor,anchor=base east,inner sep=0pt, outer sep=0pt, scale=  0.70] at ( 34.96,207.44) {100};
\end{scope}
\begin{scope}
\path[clip] ( 43.36, 28.91) rectangle (239.94,216.81);
\definecolor[named]{fillColor}{rgb}{0.83,0.83,0.83}

\path[fill=fillColor] ( 50.64,133.22) --
	( 52.46,133.55) --
	( 54.28,133.88) --
	( 56.10,134.22) --
	( 57.92,134.56) --
	( 59.74,134.91) --
	( 61.56,135.26) --
	( 63.38,135.61) --
	( 65.20,135.97) --
	( 67.02,136.33) --
	( 68.84,136.70) --
	( 70.66,137.08) --
	( 72.48,137.45) --
	( 74.30,137.84) --
	( 76.12,138.23) --
	( 77.94,138.62) --
	( 79.76,139.02) --
	( 81.58,139.42) --
	( 83.40,139.83) --
	( 85.23,140.25) --
	( 87.05,140.67) --
	( 88.87,141.09) --
	( 90.69,141.53) --
	( 92.51,141.97) --
	( 94.33,142.41) --
	( 96.15,142.86) --
	( 97.97,143.32) --
	( 99.79,143.78) --
	(101.61,144.25) --
	(103.43,144.73) --
	(105.25,145.21) --
	(107.07,145.70) --
	(108.89,146.20) --
	(110.71,146.70) --
	(112.53,147.22) --
	(114.35,147.73) --
	(116.17,148.26) --
	(117.99,148.80) --
	(119.81,149.34) --
	(121.63,149.89) --
	(123.45,150.45) --
	(125.27,151.02) --
	(127.09,151.59) --
	(128.91,152.18) --
	(130.73,152.77) --
	(132.55,153.38) --
	(134.37,153.99) --
	(136.19,154.61) --
	(138.01,155.24) --
	(139.83,155.88) --
	(141.65,156.54) --
	(143.47,157.20) --
	(145.29,157.87) --
	(147.11,158.55) --
	(148.93,159.25) --
	(150.75,159.96) --
	(152.57,160.67) --
	(154.39,161.40) --
	(156.21,162.15) --
	(158.03,162.90) --
	(159.85,163.67) --
	(161.67,164.45) --
	(163.49,165.24) --
	(165.31,166.05) --
	(167.13,166.87) --
	(168.95,167.71) --
	(170.77,168.56) --
	(172.59,169.43) --
	(174.41,170.31) --
	(176.23,171.20) --
	(178.05,172.12) --
	(179.87,173.05) --
	(181.69,174.00) --
	(183.51,174.96) --
	(185.33,175.94) --
	(187.15,176.95) --
	(188.97,177.97) --
	(190.79,179.01) --
	(192.61,180.07) --
	(194.43,181.15) --
	(196.25,182.25) --
	(198.07,183.37) --
	(199.89,184.52) --
	(201.71,185.69) --
	(203.53,186.88) --
	(205.35,188.10) --
	(207.17,189.34) --
	(208.99,190.61) --
	(210.81,191.90) --
	(212.63,193.23) --
	(214.45,194.58) --
	(216.27,195.96) --
	(218.09,197.37) --
	(219.91,198.81) --
	(221.74,200.28) --
	(223.56,201.79) --
	(225.38,203.33) --
	(227.20,204.90) --
	(229.02,206.51) --
	(230.84,208.16) --
	(232.66,209.85) --
	(232.66,209.85) --
	(230.84,208.16) --
	(229.02,206.49) --
	(227.20,204.84) --
	(225.38,203.23) --
	(223.56,201.63) --
	(221.74,200.06) --
	(219.91,198.52) --
	(218.09,196.99) --
	(216.27,195.49) --
	(214.45,194.01) --
	(212.63,192.56) --
	(210.81,191.12) --
	(208.99,189.70) --
	(207.17,188.31) --
	(205.35,186.93) --
	(203.53,185.57) --
	(201.71,184.23) --
	(199.89,182.91) --
	(198.07,181.61) --
	(196.25,180.33) --
	(194.43,179.06) --
	(192.61,177.81) --
	(190.79,176.58) --
	(188.97,175.36) --
	(187.15,174.16) --
	(185.33,172.97) --
	(183.51,171.80) --
	(181.69,170.65) --
	(179.87,169.51) --
	(178.05,168.38) --
	(176.23,167.27) --
	(174.41,166.17) --
	(172.59,165.09) --
	(170.77,164.02) --
	(168.95,162.96) --
	(167.13,161.91) --
	(165.31,160.88) --
	(163.49,159.86) --
	(161.67,158.86) --
	(159.85,157.86) --
	(158.03,156.88) --
	(156.21,155.91) --
	(154.39,154.95) --
	(152.57,154.00) --
	(150.75,153.06) --
	(148.93,152.13) --
	(147.11,151.22) --
	(145.29,150.31) --
	(143.47,149.41) --
	(141.65,148.53) --
	(139.83,147.65) --
	(138.01,146.79) --
	(136.19,145.93) --
	(134.37,145.08) --
	(132.55,144.25) --
	(130.73,143.42) --
	(128.91,142.60) --
	(127.09,141.79) --
	(125.27,140.99) --
	(123.45,140.19) --
	(121.63,139.41) --
	(119.81,138.63) --
	(117.99,137.86) --
	(116.17,137.10) --
	(114.35,136.35) --
	(112.53,135.60) --
	(110.71,134.87) --
	(108.89,134.14) --
	(107.07,133.42) --
	(105.25,132.70) --
	(103.43,131.99) --
	(101.61,131.29) --
	( 99.79,130.60) --
	( 97.97,129.91) --
	( 96.15,129.23) --
	( 94.33,128.56) --
	( 92.51,127.89) --
	( 90.69,127.23) --
	( 88.87,126.58) --
	( 87.05,125.93) --
	( 85.23,125.29) --
	( 83.40,124.66) --
	( 81.58,124.03) --
	( 79.76,123.40) --
	( 77.94,122.79) --
	( 76.12,122.18) --
	( 74.30,121.57) --
	( 72.48,120.97) --
	( 70.66,120.38) --
	( 68.84,119.79) --
	( 67.02,119.20) --
	( 65.20,118.62) --
	( 63.38,118.05) --
	( 61.56,117.48) --
	( 59.74,116.92) --
	( 57.92,116.36) --
	( 56.10,115.81) --
	( 54.28,115.26) --
	( 52.46,114.72) --
	( 50.64,114.18) --
	cycle;
\definecolor[named]{drawColor}{rgb}{0.00,0.00,0.00}

\path[draw=drawColor,line width= 0.4pt,line join=round,line cap=round] ( 50.64,126.40) --
	( 52.46,126.84) --
	( 54.28,127.28) --
	( 56.10,127.73) --
	( 57.92,128.18) --
	( 59.74,128.64) --
	( 61.56,129.10) --
	( 63.38,129.56) --
	( 65.20,130.03) --
	( 67.02,130.50) --
	( 68.84,130.97) --
	( 70.66,131.45) --
	( 72.48,131.94) --
	( 74.30,132.43) --
	( 76.12,132.92) --
	( 77.94,133.42) --
	( 79.76,133.93) --
	( 81.58,134.44) --
	( 83.40,134.95) --
	( 85.23,135.47) --
	( 87.05,135.99) --
	( 88.87,136.52) --
	( 90.69,137.05) --
	( 92.51,137.59) --
	( 94.33,138.14) --
	( 96.15,138.69) --
	( 97.97,139.25) --
	( 99.79,139.81) --
	(101.61,140.38) --
	(103.43,140.95) --
	(105.25,141.53) --
	(107.07,142.12) --
	(108.89,142.71) --
	(110.71,143.31) --
	(112.53,143.92) --
	(114.35,144.53) --
	(116.17,145.15) --
	(117.99,145.78) --
	(119.81,146.41) --
	(121.63,147.05) --
	(123.45,147.70) --
	(125.27,148.35) --
	(127.09,149.02) --
	(128.91,149.69) --
	(130.73,150.37) --
	(132.55,151.06) --
	(134.37,151.75) --
	(136.19,152.46) --
	(138.01,153.17) --
	(139.83,153.89) --
	(141.65,154.62) --
	(143.47,155.36) --
	(145.29,156.11) --
	(147.11,156.87) --
	(148.93,157.64) --
	(150.75,158.42) --
	(152.57,159.21) --
	(154.39,160.01) --
	(156.21,160.82) --
	(158.03,161.64) --
	(159.85,162.47) --
	(161.67,163.32) --
	(163.49,164.17) --
	(165.31,165.04) --
	(167.13,165.92) --
	(168.95,166.82) --
	(170.77,167.72) --
	(172.59,168.64) --
	(174.41,169.58) --
	(176.23,170.53) --
	(178.05,171.49) --
	(179.87,172.47) --
	(181.69,173.46) --
	(183.51,174.47) --
	(185.33,175.49) --
	(187.15,176.54) --
	(188.97,177.60) --
	(190.79,178.67) --
	(192.61,179.77) --
	(194.43,180.88) --
	(196.25,182.01) --
	(198.07,183.16) --
	(199.89,184.34) --
	(201.71,185.53) --
	(203.53,186.75) --
	(205.35,187.98) --
	(207.17,189.25) --
	(208.99,190.53) --
	(210.81,191.84) --
	(212.63,193.18) --
	(214.45,194.54) --
	(216.27,195.93) --
	(218.09,197.35) --
	(219.91,198.79) --
	(221.74,200.27) --
	(223.56,201.78) --
	(225.38,203.32) --
	(227.20,204.90) --
	(229.02,206.51) --
	(230.84,208.16) --
	(232.66,209.85);

\path[draw=drawColor,line width= 0.4pt,dash pattern=on 1pt off 3pt ,line join=round,line cap=round] ( 50.84,  0.00) --
	( 52.46, 61.39) --
	( 54.28, 67.55) --
	( 56.10, 71.86) --
	( 57.92, 75.30) --
	( 59.74, 78.25) --
	( 61.56, 80.87) --
	( 63.38, 83.24) --
	( 65.20, 85.43) --
	( 67.02, 87.48) --
	( 68.84, 89.41) --
	( 70.66, 91.25) --
	( 72.48, 93.01) --
	( 74.30, 94.70) --
	( 76.12, 96.33) --
	( 77.94, 97.91) --
	( 79.76, 99.44) --
	( 81.58,100.94) --
	( 83.40,102.40) --
	( 85.23,103.83) --
	( 87.05,105.23) --
	( 88.87,106.61) --
	( 90.69,107.96) --
	( 92.51,109.30) --
	( 94.33,110.61) --
	( 96.15,111.91) --
	( 97.97,113.19) --
	( 99.79,114.46) --
	(101.61,115.72) --
	(103.43,116.97) --
	(105.25,118.21) --
	(107.07,119.43) --
	(108.89,120.65) --
	(110.71,121.87) --
	(112.53,123.07) --
	(114.35,124.27) --
	(116.17,125.46) --
	(117.99,126.65) --
	(119.81,127.84) --
	(121.63,129.02) --
	(123.45,130.20) --
	(125.27,131.38) --
	(127.09,132.55) --
	(128.91,133.73) --
	(130.73,134.90) --
	(132.55,136.07) --
	(134.37,137.24) --
	(136.19,138.41) --
	(138.01,139.59) --
	(139.83,140.76) --
	(141.65,141.93) --
	(143.47,143.11) --
	(145.29,144.29) --
	(147.11,145.47) --
	(148.93,146.65) --
	(150.75,147.84) --
	(152.57,149.03) --
	(154.39,150.22) --
	(156.21,151.42) --
	(158.03,152.62) --
	(159.85,153.83) --
	(161.67,155.04) --
	(163.49,156.26) --
	(165.31,157.48) --
	(167.13,158.71) --
	(168.95,159.94) --
	(170.77,161.18) --
	(172.59,162.43) --
	(174.41,163.69) --
	(176.23,164.95) --
	(178.05,166.22) --
	(179.87,167.50) --
	(181.69,168.78) --
	(183.51,170.08) --
	(185.33,171.38) --
	(187.15,172.70) --
	(188.97,174.02) --
	(190.79,175.35) --
	(192.61,176.70) --
	(194.43,178.05) --
	(196.25,179.42) --
	(198.07,180.80) --
	(199.89,182.19) --
	(201.71,183.59) --
	(203.53,185.01) --
	(205.35,186.44) --
	(207.17,187.88) --
	(208.99,189.34) --
	(210.81,190.81) --
	(212.63,192.30) --
	(214.45,193.80) --
	(216.27,195.33) --
	(218.09,196.86) --
	(219.91,198.42) --
	(221.74,199.99) --
	(223.56,201.58) --
	(225.38,203.20) --
	(227.20,204.83) --
	(229.02,206.48) --
	(230.84,208.15) --
	(232.66,209.85);

\node[text=drawColor,anchor=base,inner sep=0pt, outer sep=0pt, scale=  0.55] at ( 50.64,118.73) {+};

\node[text=drawColor,anchor=base,inner sep=0pt, outer sep=0pt, scale=  0.55] at ( 59.74,121.12) {+};

\node[text=drawColor,anchor=base,inner sep=0pt, outer sep=0pt, scale=  0.55] at ( 68.84,124.36) {+};

\node[text=drawColor,anchor=base,inner sep=0pt, outer sep=0pt, scale=  0.55] at ( 77.94,127.71) {+};

\node[text=drawColor,anchor=base,inner sep=0pt, outer sep=0pt, scale=  0.55] at ( 87.05,130.55) {+};

\node[text=drawColor,anchor=base,inner sep=0pt, outer sep=0pt, scale=  0.55] at ( 96.15,134.09) {+};

\node[text=drawColor,anchor=base,inner sep=0pt, outer sep=0pt, scale=  0.55] at (105.25,137.11) {+};

\node[text=drawColor,anchor=base,inner sep=0pt, outer sep=0pt, scale=  0.55] at (114.35,140.39) {+};

\node[text=drawColor,anchor=base,inner sep=0pt, outer sep=0pt, scale=  0.55] at (123.45,144.11) {+};

\node[text=drawColor,anchor=base,inner sep=0pt, outer sep=0pt, scale=  0.55] at (132.55,147.98) {+};

\node[text=drawColor,anchor=base,inner sep=0pt, outer sep=0pt, scale=  0.55] at (141.65,152.07) {+};

\node[text=drawColor,anchor=base,inner sep=0pt, outer sep=0pt, scale=  0.55] at (150.75,155.78) {+};

\node[text=drawColor,anchor=base,inner sep=0pt, outer sep=0pt, scale=  0.55] at (159.85,160.21) {+};

\node[text=drawColor,anchor=base,inner sep=0pt, outer sep=0pt, scale=  0.55] at (168.95,165.08) {+};

\node[text=drawColor,anchor=base,inner sep=0pt, outer sep=0pt, scale=  0.55] at (178.05,169.61) {+};

\node[text=drawColor,anchor=base,inner sep=0pt, outer sep=0pt, scale=  0.55] at (187.15,174.75) {+};

\node[text=drawColor,anchor=base,inner sep=0pt, outer sep=0pt, scale=  0.55] at (196.25,180.82) {+};

\node[text=drawColor,anchor=base,inner sep=0pt, outer sep=0pt, scale=  0.55] at (205.35,186.63) {+};

\node[text=drawColor,anchor=base,inner sep=0pt, outer sep=0pt, scale=  0.55] at (214.45,193.56) {+};

\node[text=drawColor,anchor=base,inner sep=0pt, outer sep=0pt, scale=  0.55] at (223.56,200.35) {+};

\node[text=drawColor,anchor=base,inner sep=0pt, outer sep=0pt, scale=  0.55] at (232.66,208.45) {+};
\end{scope}
\begin{scope}
\path[clip] (  0.00,  0.00) rectangle (254.39,231.26);
\definecolor[named]{drawColor}{rgb}{0.00,0.00,0.00}

\path[draw=drawColor,line width= 0.4pt,line join=round,line cap=round] ( 43.36, 28.91) --
	(239.94, 28.91) --
	(239.94,216.81) --
	( 43.36,216.81) --
	( 43.36, 28.91);

\node[text=drawColor,rotate= 90.00,anchor=base,inner sep=0pt, outer sep=0pt, scale=  0.70] at ( 16.96,122.86) {$\lambda$};

\node[text=drawColor,anchor=base,inner sep=0pt, outer sep=0pt, scale=  0.70] at (141.65,  7.31) {$\gamma$};
\end{scope}
\end{tikzpicture}
\caption{\label{fig:Strauss100} Intensities of two-dimensional Strauss processes with $\beta=100$, $r=0.05$ and values of $\gamma$ ranging from $0$ to $1$. The solid line is $\lambda_{PS}$, the dashed line is $\lambda_{MF}$, and the grey area corresponds to the bounds in \eqref{eq:lambdaPIP}. The pluses are estimates of the intensities based on 10,000 simulations each.}
\end{center}
\end{figure}

For a comparison of our bounds on the intensity to known approximations from the literature we concentrate on two methods.

The first one is the \emph{Poisson-saddlepoint} approximation proposed in \citep{baddeleynair12}. The authors replaced in \eqref{eq:lambda-pip} the Gibbs process $\Xi\sim\pip(\beta,\varphi)$ by a Poisson process $\Eta_{\lambda_{PS}}$ with intensity $\lambda_{PS}$ such that the following equality holds
\begin{equation}
\label{eq:ps}
\lambda_{PS}=\E\lambda(0\mvert \Eta_{\lambda_{PS}})=\beta \hbit \E\Bigl(\prod_{y\in \Eta_{\lambda_{PS}}}\varphi(y)\Bigr).
\end{equation}
Solving this equation yields
\begin{equation}
\label{eq:lambda-PS}
\lambda_{PS}=\frac{W(\beta G)}{G},
\end{equation}
where $W$ is \emph{Lambert's W function}, the inverse of $x\mapsto xe^x$, and $G=\int_{\R^d}1-\varphi(x)\,dx$ as above. 

The second method is the \emph{mean-field} approximation that was also described in \citep{baddeleynair12} and is given by
\begin{equation}
\label{eq:lambda-MF}
\lambda_{MF}=\frac{W(\beta \Gamma)}{\Gamma},
\end{equation}
where $\Gamma=-\int_{\R^d}\log(\varphi(x))\,dx$. Figure~\ref{fig:Strauss100} shows the two approximations and our bounds from Inequality~\eqref{eq:lambdaPIP} for two-dimensional Strauss processes.

In \citep{baddeleynair12} it is shown that under the conditions of Theorem~\ref{thm:lambdaPIP} we have $\lambda\ge \lambda_{MF}$. The authors also conjectured, based on simulations for Strauss processes, that $\lambda_{PS}$ is an upper bound for $\lambda$. However, the next example indicates that this is not generally true.

\begin{figure}[!h]
\begin{center}

\caption{\label{fig:ringprozess} The two processes in Example~\ref{ex:ringprozess} restricted to the unit square. In the left panel is a realisation of the hard annulus process having $489$ points, whereas $\lambda_{PS}=295.2$. In the right panel is a hard core process realisation having $188$ points.}
\end{center}
\end{figure}

\begin{example}
\label{ex:ringprozess}
Consider the process $\Xi\sim \pip(\beta,\varphi)$ with the interaction function
\begin{equation*}
\varphi(x) =\begin{cases} 1 \quad &\text{if} \quad  \|x\| \le r\\
0 \quad &\text{if} \quad r<\|x\| \le R\\
1 \quad &\text{if} \quad  \|x\| > R
 \end{cases}
\end{equation*}
for constants $0\le r\le R $. We refer to this as a \emph{hard annulus process}. It is a special case of a so-called \emph{multiscale Strauss process}, see \citep[Ex. 6.2]{moellerwaage04}. Let $d=2$, $\beta=3000$, $r=0.05$ and $R=\sqrt{2}r$. 
Then 
\begin{equation*}
\lambda_{MF}=0,\quad \frac{\beta}{1+\beta G}=122.1, \quad \lambda_{PS}=295.2\quad \text{and}\quad \frac{\beta}{2-e^{-\beta G}}=1500.
\end{equation*}
An estimate of the intensity based on $300$ simulations gave $\hat{\lambda}=493.8 > \lambda_{PS}$. For comparison we also estimated the intensity of a Strauss hard core process with the same $\beta$ and $G$ and obtained $\hat{\lambda}=193.3$. Figure~\ref{fig:ringprozess} shows that although the two processes have the same $\beta$ and $G$, their realisations look quite different. All simulations were performed by long runs ($10^7$ steps) of Markov Chain Monte Carlo. 

We were not able to prove that $\lambda>\lambda_{PS}$ in this case, but bring forward the following heuristic argument for the observed phenomenon. The simulations show that for large $\beta$ the points tend to cluster on ``islands'' of radius $\le r/2$ which are separated by a distance $\ge R$. Since the points within each island do not interact, we expect the intensity to grow linearly in $\beta$ for large $\beta$. However $\lambda_{PS}$ only grows logarithmically for large $\beta$, so that at some point the intensity will overtake.
\end{example}

Even if $\lambda_{PS}$ may not serve as a bound on $\lambda$ it remains useful as an approximation. Empirically its values stay relatively close to the simulated values, whereas the difference of our upper and lower bounds in \eqref{eq:lambdaPIP} increases for large $\beta G$.

The following result in connection with Theorem~\ref{thm:lambdaPIP} gives an upper bound on the error in Poisson saddlepoint approximation.

\begin{proposition}
\label{lemma:PS}
Under the conditions of Theorem~\ref{thm:lambdaPIP} we have
\begin{equation}
\label{eq:bounds-lambda-PS}
\frac{\beta}{1+\beta G}\le \lambda_{PS} \le \frac{\beta}{2-e^{-\beta G}}.
\end{equation}
\end{proposition}

\begin{proof}
Since $\lambda_{PS}=W(\beta G)/G$, it suffices to show the following two inequalities:
\begin{equation}
\label{eq:W}
\frac{x}{1+x}\le W(x) \quad \text{and} \quad
W(x) \le \frac{x}{2-e^{-x}} 
\end{equation}
for all $x\ge0$. The first one follows from $x/(1+x)\le \log(1+x)$, see \citep[Eq.~4.1.33]{as64}, by transforming it to
\begin{equation*}
  \frac{x}{1+x} \exp \bigl( \frac{x}{1+x} \bigr) \leq x
\end{equation*}
and applying the increasing function $W$ on both sides. For the second inequality note that
\begin{equation}
\log(2-e^{-x})\le \frac{x}{2-e^{-x}}.
\end{equation}
This holds because we have equality for $x=0$ and it is straightforward to see that the derivative of the left hand side is less than or equal to the derivative of the right hand side for all $x\ge 0$. A similar transformation as above and applying $W$ on both sides again gives the second inequality in~\eqref{eq:W}.
\end{proof}

\section{Summary statistics}
\label{sec:s-stat}

For a stationary point process $\Xi$ the \emph{empty space function} or \emph{$F$ function} is defined as the cumulative distribution function of the distance from the origin to the nearest point in $\Xi$, i.e.
\begin{align*}
F(t)&=\P(\exists y\in \Xi \colon \|y\| \le t)=1-\P(\Xi(\BB(0,t))=0)\\
&=1-\E\Bigl(\prod_{y\in\Xi}\one\{y\notin \BB(0,t)\}\Bigr)=1-\Psi_\Xi(\one\{\cdot \notin \BB(0,t)\}),
\end{align*}
where $\BB(x,t)$ denotes the closed ball centred at $x\in\R^d$ with radius $t\ge0$. Thus for a locally stable Gibbs process $\Xi$ with constant $c^*$ we obtain from Theorem~\ref{thm:bounds}
\begin{equation}
\label{eq:boundsF}
\frac{\lambda}{c^*}\big(1-\exp(-c^*\alpha_dt^d)\big)\le F(t) \le \lambda \alpha_d t^d.
\end{equation}
Note that for a Poisson process with intensity $\lambda$ we may choose $c^*=\lambda$, in which case the lower bound in \eqref{eq:boundsF} is exact. A minor drawback of the bounds in \eqref{eq:boundsF} is that the intensity is in general not known and has to be estimated as well, e.g.\ by the methods of Section~\ref{sec:bounds-int}.

The \emph{nearest neighbour function} or \emph{$G$ function} is defined as the cumulative distribution function of the distance from a typical point of $\Xi$ (in the sense of the Palm distribution) to its nearest neighbour. For pairwise interaction processes $\Xi\sim\pip(\beta,\varphi)$ the $G$ function is computed in \citep[Sec.~5]{mase90} as
\begin{equation}
\label{eq:G}
G(t)=1-\frac{\beta}{\lambda}\E\Big(\prod_{y\in \Xi} \one\{y\notin \BB(0,t)\}\varphi(y)\Big)=1-\frac{\beta}{\lambda}\Psi_\Xi\big(\one\{\cdot \notin \BB(0,t)\}\varphi(\cdot)\big).
\end{equation}
Thus if $\Xi$ is inhibitory and has finite interaction range, setting $c^*=\beta$ in Theorem~\ref{thm:bounds} yields  
\begin{equation}
\label{eq:boundsG}
2-\frac{\beta}{\lambda}-\exp(-\beta \tilde{G}_t)\le G(t) \le 1-\frac{\beta}{\lambda}+\beta\tilde{G}_t,
\end{equation}
where $\tilde{G}_t=\int_{\R^d} [ 1-\varphi(x)\one\{\|x\|>t\} ] \,dx$. The left panel of Figure~\ref{fig:StraussK} shows these bounds for the hard annulus process from Example~\ref{ex:ringprozess} with parameters $\beta=70$, $r=0.025$ and $R=0.035$.

Let us furthermore assume that $\Xi$ is isotropic. Then the \emph{$K$ function} is defined as
\begin{equation*}
K(t)=\alpha_dd\int_0^ts^{d-1}\rho(s)\;ds,
\end{equation*}
where $\rho$ is the pair correlation function. By \eqref{eq:pcf}, \eqref{eq:corr-fun} and Theorem~\ref{thm:bounds} we obtain bounds on $\rho$ as
\begin{equation}
\label{eq:boundspcf}
\varphi(x)\biggl(\frac{\beta^2}{\lambda^2}-\frac{\beta^2\tilde{G}_x}{\lambda}\biggr)\le \rho(\|x\|) \le \varphi(x)\biggl(\frac{\beta^2}{\lambda^2}-\frac{\beta}{\lambda}\big(1-\exp(-\beta\tilde{G}_x)\big)\biggr),
\end{equation}
where $\tilde{G}_x=\int_{\R^d} [ 1-\varphi(y)\varphi(y-x) ] \;dy$, and (in most cases numeric) integration of \eqref{eq:boundspcf} yields bounds on the $K$ function.

\begin{example}
\label{ex:StraussK}
Let $\Xi$ be a Strauss process in two dimensions. Then
\begin{equation}
\tilde{G}_x=2\pi r^2(1-\gamma)-2 r^2 (1-\gamma)^2\Biggr(\mathrm{arccos}\biggl(\frac{\|x\|}{2r}\biggr)-\frac{\|x\|}{2r}\sqrt{1-\Bigl(\frac{\|x\|}{2r}\Bigr)^2}\Biggr);
\end{equation}
see also \citep{bn12}.
Since we do not know the true intensity of the Strauss process, we plug in the bounds of \eqref{eq:lambdaPIP} into \eqref{eq:boundspcf} to obtain bounds on the $K$ function. This procedure causes twice an error and therefore the estimates on $K$ are good only for smaller values of $\beta G$.
The right panel of Figure~\ref{fig:StraussK} shows these estimates for $\beta=40$, $r=0.05$ and $\gamma=0$.

\begin{figure}[ht]
\begin{center}

\caption{\label{fig:StraussK} \emph{Left:} $G$ function of the hard annulus process from Example~\ref{ex:ringprozess} with parameters $\beta=70$, $r=0.025$, $R=0.035$. \emph{Right:} $K$ function of a Strauss hard core process with parameters $\beta=40$, $r=0.05$. 
In both panels the solid line is an estimate of the true function based on 1,000 simulations, the dashed line is the  true function for a Poisson process (in the case of the $G$ function with the same intensity as the one obtained by simulation for the hard annulus process), and the grey area corresponds to the bounds computed in the text.}
\end{center}
\end{figure}

\end{example}

\section{Proof of Theorem~\ref{thm:bounds}}
\label{sec:proofs}

The main strategy of the proof is to replace in \eqref{eq:pgfl} the process $\Xi$ by a Poisson process $\Eta$, and then use Stein's method to bound the error
\begin{equation}
\label{eq:error}
\E\Big( \prod_{y\in \Xi} g(y) \Big)-\E\Big( \prod_{y\in \Eta} g(y) \Big).
\end{equation}

In the context of Poisson process approximation, Stein's method for bounding expressions of the form $\abs{\E f(\Xi)-\E f(\Eta)}$ uniformly in $f$ from a class of functions $\mcf$ was first introduced in \citep{bb92}. A nice exposition of the main body of the method with various ramifications and detailed proofs can be found in \citep{xia05}. More recent developments include the generalisation to Gibbs process approximation \cite{ss12}.

The central idea of Stein's method is to set up an equation of the form
\begin{equation}
\label{eq:stein}
  f(\xi) - \EE f(\Eta) = \msa h_f(\xi) \quad \text{for all $\xi \in \mfn$},
\end{equation}
where $\msa$ is an operator on a space of functions $h \colon \mfn \to \RR$ that characterises the distribution $Q$ of $\Eta$ in the sense that $\widetilde{\Eta}$ is $Q$-distributed if and only if $\EE \msa h(\widetilde{\Eta}) = 0$ for all functions~$h$. The hope is that a suitable choice of $\msa$ results in a right hand side whose expectation is much easier to bound. A particularly successful approach is the generator method by Barbour \citep{barbour88}, which suggests to choose for $\msa$ the infinitesimal generator of a Markov process whose stationary distribution is $Q$.

In the context of an approximation based on a Poisson process $\Eta_{\nu}$ with intensity $\nu > 0$ on a compact set $A$, the default choice is a spatial immigration-death process on $A$ with immigration rate $\nu$ and unit per-capita death rate.
This is a pure-jump Markov process on $\mfn \vert_A$ that holds any state $\eta \in \mfn \vert_A$ for an exponentially distributed time with mean $1/(\nu \abs{A} + \eta(A))$, where $\abs{A}$ denotes the Lebesgue measure of $A$; then a uniformly distributed point in $A$ is added with probability $\nu \abs{A}/(\nu \abs{A} + \eta(A))$, or a uniformly distributed point in $\eta$ is deleted with probability $\eta(A)/(\nu \abs{A} + \eta(A))$. Let $Z_{\xi}$ be such a process started at configuration $\xi$.
The operator $\msa$ then takes the form
\begin{equation}
\label{eq:generator}
  \msa h(\xi) = \int_{A} \bigl[ h(\xi+\delta_x)-h(\xi) \bigr]
  \hbit \nu \; dx+ \int_{A} \bigl[ h(\xi-\delta_x)-h(\xi) \bigr] \; \xi(d x).
\end{equation}

Immigration-death processes have many nice properties. In particular it is known that $Z_{\emptyset}(t)$ is a Poisson process with intensity $\nu (1-e^{-t})$ for every $t \geq 0$. Let $\{E_x\}_{x\in \xi}$ be i.i.d.\ exponentially distributed random variables with mean one and introduce the death process $D_\xi(t)=\sum_{x\in\xi}\one\{E_x>t\}\delta_x$.
Constructing $Z_{\emptyset}$ and $D_{\xi}$ independently on the same probability space, $Z_\xi$ can be represented as $Z_\xi(t) \eqinlaw Z_\emptyset(t)+D_\xi(t)$ for every $t \geq 0$; see \citep[Thm.~3.5]{xia05}.

We can then solve the Stein equation~\eqref{eq:stein}.
Taking $g \colon \RR^d \to [0,1]$ such that $A = \supp(1-g)$ is compact, consider the function $f \colon \mfn \to [0,1]$,
\begin{equation}
\label{eq:f}
f(\xi)=f(\xi \vert_A)=\prod_{y\in\xi}g(y).
\end{equation}

By \citep[Thm.~5.2]{xia05} the function $h_f \colon \mfn \to \RR$,
\begin{equation}
\label{eq:stein-sol}
h_f(\xi)=h_f(\xi \vert_A) = -\int_0^\infty \big[\E \bigl(f(Z_{\xi\vert_A}(t))\bigr)-\E(f(\Eta_\nu))\big]\;dt,
\end{equation}
is well-defined and satisfies~\eqref{eq:stein}. Thus
\begin{align}
\E f(\Xi)-\E f(\Eta_\nu) &= \E \msa h_f(\Xi) \nonumber \\
&= \E \int_{A} \bigl[ h_f(\Xi+\delta_x)-h_f(\Xi) \bigr]
\hbit \nu \; dx+ \E \int_{A} \bigl[ h_f(\Xi-\delta_x)-h_f(\Xi) \bigr] \; \Xi(d x) \nonumber \\
&= \E\int_{A}\big[h_f(\Xi+\delta_x)-h_f(\Xi)\big]\big( \nu-\lambda(x\mvert \Xi)\big)\; dx \nonumber \\
&= \E\int_{\RR^d}\big[h_f(\Xi+\delta_x)-h_f(\Xi)\big]\big( \nu-\lambda(x\mvert \Xi)\big)\; dx,
\label{eq:SteinGNZ}
\end{align}
where we applied the Georgii--Nguyen--Zessin equation on $\RR^d$ to the function $\bigl[(x,\xi) \mapsto \one_A(x) \bigl( h_f(\xi)-h_f(\xi+\delta_x) \bigr)\bigr]$ for obtaining the second equality.

Equation~\eqref{eq:SteinGNZ} is our starting point for further considerations.

\begin{proposition}
\label{prop:1}
Let $\Xi$ be a stationary Gibbs process with intensity $\lambda$ and conditional intensity $\lambda(\cdot\mvert \cdot)$. Let $g\colon {\R^d} \to [0,1]$ be a function such that $1-g$ has bounded support.  Then for all $\nu>0$
\begin{equation}
\label{eq:prop1}
\E\Big(\prod_{y\in\Xi}g(y)\Big)=1-\frac{\lambda}{\nu} \big(1-e^{-\nu G}\big)+I_\nu(g),
\end{equation}
where 
\begin{equation}
\label{eq:I}
I_\nu(g)=e^{-\nu G}\hbit\E\biggl( \int_0^1e^{\nu Gs}\Big(1-\prod_{y\in \Xi}\big(1-s(1-g(y))\big)\Big)\;ds \hspace*{0.4em} \cdot \, \int_{\R^d}(1-g(x))(\lambda(x\mvert \Xi)-\nu)\;dx\biggr).
\end{equation}
\end{proposition}

\begin{proof}
It is well known that for the Poisson process $\Eta_{\nu}$ we have
\begin{equation}
\label{eq:Pois-pgfl}
\E \Big(\prod_{y\in \Eta_\nu}g(y)\Big)=\exp\bigg(-\nu\int_{\R^d}1-g(x)\; dx\bigg)=e^{-\nu G};
\end{equation}
see e.g.\ \citep[Eq.~9.4.17]{dvj08}.
We then follow the main proof strategy laid out above in order to re-express
\begin{equation*}
  \E\Big(\prod_{y\in \Xi}g(y)\Big)-e^{-\nu G}.
\end{equation*}

Starting from Equation~\eqref{eq:SteinGNZ}, we may use the decomposition $Z_{\xi+\delta_x} \eqinlaw Z_\xi+D_{\delta_x}$ with independent $Z_\xi$ and $D_{\delta_x}$ to see that for any $\xi \in \mfn\vert_A$
\begin{align*}
h_f(\xi+\delta_x)-h_f(\xi)&=\int_0^\infty \E f(Z_\xi(t))-\E f(Z_{\xi+\delta_x}(t))\;dt \\
&=\int_0^\infty \E f(Z_\xi(t))-\E f(Z_{\xi}(t)+D_{\delta_x}(t))\;dt \\
&=(1-g(x))\int_0^\infty \E f(Z_\xi(t)) \hbit \P(D_{\delta_x}(t) \neq \emptyset)\;dt \\
&=(1-g(x))\int_0^\infty \E f(Z_\xi(t)) \hbit e^{-t}\;dt \\
&=(1-g(x))\int_0^\infty \E \big(f(Z_\emptyset(t))+(f(Z_\xi(t)) - f(Z_\emptyset(t))\big)  \hbit e^{-t}\;dt.
\end{align*} 
By further decomposing $Z_\xi \eqinlaw Z_\emptyset+D_\xi$ with independent $Z_{\emptyset}$ and $D_\xi$, we obtain
\begin{align*}
\E \big(f(Z_\xi(t)) - f(Z_\emptyset(t))\big)&=\E\Big(\prod_{y\in Z_\emptyset(t)}g(y)\Big)\hbit\E\Big(\prod_{y\in D_\xi(t)}g(y)-1\Big)\\
&=\exp\big(-\nu(1-e^{-t})G\big)\Big(\prod_{y\in \xi}\bigl(1-e^{-t}(1-g(y))\bigr)-1 \Big),
\end{align*}
where for the first expectation we used \eqref{eq:Pois-pgfl} and that $Z_\emptyset(t)$ is a Poisson process with intensity $\nu(1-e^{-t})$; for the second expectation note that each point of $\xi$ survives independently up to time $t$ with probability $e^{-t}$. Thus in total
\begin{equation*}
\begin{split}
h_f(&\xi+\delta_x)-h_f(\xi)\\[1mm]
&=(1-g(x))e^{-\nu G}\int_0^\infty\Bigl[ e^{\nu Ge^{-t}} + e^{\nu Ge^{-t}}\Bigl(\prod_{y\in \xi}\bigl(1-e^{-t}(1-g(y))\bigr) -1\Bigr)\Bigr]e^{-t}\;dt \\
&= (1-g(x))\frac{1-e^{-\nu G}}{\nu G} + (1-g(x)) e^{-\nu G}\int_0^1 e^{\nu Gs}\Bigl(\prod_{y\in \xi}\bigl(1-s(1-g(y))\bigr) -1 \Bigr)\;ds
\end{split}
\end{equation*}
by the substitution $s=e^{-t}$.

Plugging this into Equation~\eqref{eq:SteinGNZ} and using $\E \lambda(x\mvert \Xi)=\lambda$ finally yields
\begin{equation*}
\E\Big(\prod_{y\in \Xi}g(y)\Big)-e^{-\nu G}=\frac{\nu-\lambda}{\nu}\big(1-e^{-\nu G}\big)+I_\nu(g).
\end{equation*}
\end{proof}

\begin{proposition}
\label{prop:bounds-I}
Let $\Xi$ be a stationary locally stable Gibbs process with constant $c^*$. Then for all $0<\nu< c^*$
\begin{equation*}
\underline{I_{\nu}}(g)\le I_\nu(g)\le \overline{I_\nu}(g),
\end{equation*}
where
\begin{align}
\underline{I_{\nu}}(g)&= -\frac{1}{c^*-\nu}\big(c^*(1-e^{-\nu G})- \nu(1-e^{-c^*G})\big) \le 0,\label{eq:boundsI1}\\
\overline{I_\nu}(g)&= \frac{1}{\nu}\big(c^*(1-e^{-\nu G})- \nu(1-e^{-c^*G})\big)\ge 0.
\label{eq:boundsI2}
\end{align}
Furthermore $I_\nu(g) \le 0$ for all $\nu \ge c^*$.
\end{proposition}

\begin{proof}
Since $-\nu \le \lambda(x\mvert \Xi)-\nu\le c^*-\nu$ a.s., we get for $\nu<c^*$ the upper bound
\begin{equation}
\label{eq:proof-b-I}
I_\nu(g)\le (c^*-\nu)e^{-\nu G}G  \int_0^1e^{\nu Gs}\Big(1-\E\prod_{y\in \Xi}\big(1-s(1-g(y))\big)\Big)\;ds , 
\end{equation}
and a similar lower bound, where $c^*-\nu$ is replaced by $-\nu$. Because $A=\supp(1-g)$ is bounded, $\Xi$ can be replaced by $\Xi\vert_A$ in \eqref{eq:proof-b-I}. It is a known fact that every locally stable Gibbs process on a bounded domain can be obtained as a dependent random thinning of a Poisson process; see \citep[Remark 3.4]{km00}. In particular, there exists a Poisson process $\Eta_{c^*}$ such that $\Xi\vert_A\subset \Eta_{c^*}$ a.s. Since $(1-s(1-g(y))\le 1$ for all $s\in [0,1]$ and for all $y\in \R^d$, we obtain
\begin{equation*}
1-\E\prod_{y\in \Xi}\bigl(1-s(1-g(y))\bigr) \le 1-\E\prod_{y\in \Eta_{c^*}}\bigl(1-s(1-g(y))\bigr) = 1-e^{-sc^*G},
\end{equation*}
where the last equality follows by \eqref{eq:Pois-pgfl}. Integrating and rearranging the terms yields the formulas \eqref{eq:boundsI1} and \eqref{eq:boundsI2}. If $\nu \ge c^*$, $I_\nu(g)$ is obviously non-positive.
\end{proof}
\vspace{4.5mm}

\noindent
\begin{proof}[Remainder of the proof of Theorem~\ref{thm:bounds}]
Propositions~\ref{prop:1} and~\ref{prop:bounds-I} yield the upper bounds
\begin{align*}
\E\Big(\prod_{y\in \Xi}g(y)\Big)&\le 1-\frac{\lambda}{\nu}\big(1-e^{-\nu G}\big)+\frac{c^*}{\nu}\big(1-e^{-\nu G}\big)-\big(1-e^{-c^* G}\big)\\
&=(c^*-\lambda)G\frac{1-e^{-\nu G}}{\nu G}+e^{-c^*G}
\end{align*}
for $0<\nu<c^*$ and 
\begin{equation*}
\E\Big(\prod_{y\in \Xi}g(y)\Big)\le 1-\lambda G\frac{1-e^{-\nu G}}{\nu G}
\end{equation*}
for $\nu \ge c^*$. Since the function $[x \mapsto (1-\exp(-x))/x]$
is monotonically decreasing for $x\ge 0$, we obtain the minimal upper bound for $\nu = c^{*}$, as
\begin{equation*}
\E\Big(\prod_{y\in \Xi}g(y)\Big)\le 1-\frac{\lambda}{c^*} \big(1-e^{-c^* G}\big).
\end{equation*}
For the lower bound recall the Weierstrass product inequality, which states
\begin{equation*}
\prod_{i=1}^n(1-a_i)\ge 1-\sum_{i=1}^na_i
\end{equation*}
for $0\le a_1,\dots,a_n\le 1$. Then, noting that the products below contain only finitely many factors $\neq 1$ by the boundedness of~$\supp(1-g)$, we have
\begin{align*}
\E \Big(\prod_{y\in \Xi}g(y)\Big)&=\E\Big(\prod_{y\in \Xi}\big(1-(1-g(y))\big)\Big)\\
&\ge 1-\E\sum_{y\in \Xi}(1-g(y))\\
&=1-\E\int_{\R^d}1-g(x)\; \Xi(dx)\\
&=1-\lambda\int_{\R^d}1-g(x)\;dx =1-\lambda G
\end{align*}
by Campell's formula; see \cite[Section~$9.5$]{dvj08}. 
\end{proof}

\begin{remark}
An alternative proof for the lower bound can be obtained by using Propositions~\ref{prop:1} and \ref{prop:bounds-I} in the analogous way as for the upper bound.
\end{remark}


\begin{thebibliography}{10}

\bibitem{as64}
Milton Abramowitz and Irene~A. Stegun.
\newblock {\em Handbook of mathematical functions with formulas, graphs, and
  mathematical tables}, volume~55 of {\em National Bureau of Standards Applied
  Mathematics Series}.
\newblock 1964.

\bibitem{bn12}
Adrian Baddeley and Gopalan Nair.
\newblock Approximating the moments of a spatial point process.
\newblock {\em Stat}, 1:18--30, 2012.

\bibitem{baddeleynair12}
Adrian Baddeley and Gopalan Nair.
\newblock Fast approximation of the intensity of {G}ibbs point processes.
\newblock {\em Electron. J. Statist.}, 6:1155--1169, 2012.

\bibitem{spatstat12}
Adrian Baddeley and Rolf Turner.
\newblock Spatstat: an {R} package for analyzing spatial point patterns.
\newblock {\em Journal of Statistical Software}, 12(6):1--42, 2005.

\bibitem{barbour88}
Andrew~D. Barbour.
\newblock Stein's method and {P}oisson process convergence.
\newblock {\em J. Appl. Probab.}, 25A:175--184, 1988.

\bibitem{bb92}
Andrew~D. Barbour and Timothy~C. Brown.
\newblock Stein's method and point process approximation.
\newblock {\em Stochastic Process. Appl.}, 43(1):9--31, 1992.

\bibitem{dvj08}
Daryl~J. Daley and David Vere-Jones.
\newblock {\em An introduction to the theory of point processes. {V}ol. {II}}.
\newblock Probability and its Applications (New York). Springer, New York,
  second edition, 2008.

\bibitem{km00}
Wilfrid~S. Kendall and Jesper M{\o}ller.
\newblock Perfect simulation using dominating processes on ordered spaces, with
  application to locally stable point processes.
\newblock {\em Adv. in Appl. Probab.}, 32(3):844--865, 2000.

\bibitem{mase90}
Shigeru Mase.
\newblock Mean characteristics of {G}ibbsian point processes.
\newblock {\em Ann. Inst. Statist. Math.}, 42(2):203--220, 1990.

\bibitem{moellerwaage04}
Jesper M{\o}ller and Rasmus~P. Waagepetersen.
\newblock {\em Statistical inference and simulation for spatial point
  processes}, volume 100 of {\em Monographs on Statistics and Applied
  Probability}.
\newblock Chapman \& Hall/CRC, Boca Raton, FL, 2004.

\bibitem{mw07}
Jesper M{\o}ller and Rasmus~P. Waagepetersen.
\newblock Modern statistics for spatial point processes.
\newblock {\em Scand. J. Statist.}, 34(4):643--684, 2007.

\bibitem{nguyenzessin79}
Xuan-Xanh Nguyen and Hans Zessin.
\newblock Integral and differential characterizations of the {G}ibbs process.
\newblock {\em Math. Nachr.}, 88:105--115, 1979.

\bibitem{r12}
{R Development Core Team}.
\newblock {\em R: A Language and Environment for Statistical Computing}.
\newblock R Foundation for Statistical Computing, Vienna, Austria, 2012.

\bibitem{ruelle69}
David Ruelle.
\newblock {\em Statistical mechanics: {R}igorous results}.
\newblock W. A. Benjamin, Inc., New York--Amsterdam, 1969.

\bibitem{ss12}
Dominic Schuhmacher and Kaspar Stucki.
\newblock On bounds for {G}ibbs point process approximation.
\newblock {\em Preprint}, 2012.
\newblock Available at http://arxiv.org/abs/1207.3096.

\bibitem{skm95}
Dietrich Stoyan, Wilfrid~S. Kendall, and Joseph Mecke.
\newblock {\em Stochastic geometry and its applications}.
\newblock Wiley Series in Probability and Mathematical Statistics. John Wiley
  \& Sons Ltd., Chichester, second edition, 1995.

\bibitem{xia05}
Aihua Xia.
\newblock Stein's method and {P}oisson process approximation.
\newblock In {\em An introduction to Stein's method}, volume~4 of {\em Lect.
  Notes Ser. Inst. Math. Sci. Natl. Univ. Singap.}, pages 115--181. Singapore
  Univ. Press, Singapore, 2005.

\end{thebibliography}
\bibliographystyle{plain}

\end{document}